\documentclass[12pt]{amsart}
\usepackage{amssymb,amsfonts,amsmath,amsthm}
\usepackage[a4paper,margin=1in]{geometry}

\usepackage{lmodern}
\usepackage{microtype}
\usepackage[english]{babel}
\usepackage[hidelinks]{hyperref} 
\usepackage{mathrsfs} 
\usepackage{mathtools}
\usepackage{enumerate}
\usepackage{xcolor}
\theoremstyle{plain} 

\newtheorem*{lemma*}{Lemma}
\newtheorem{lemma}{Lemma}
\newtheorem{proposition}{Proposition}
\newtheorem{theorem}{Theorem}

\DeclareMathOperator{\sinc}{sinc}

\newcommand{\Z}{\mathbb{Z}}

\renewcommand{\phi}{\varphi}
\begin{document} 
\title[The H\"{o}rmander--Bernhardsson extremal function: A preliminary study]
{The H\"{o}rmander--Bernhardsson extremal function: \\ A preliminary study}

\author[A. Bondarenko]{Andriy Bondarenko}
\address{Department of Mathematical Sciences \\ Norwegian University of Science and Technology \\ NO-7491 Trondheim \\ Norway}
\email{andriybond@gmail.com}

\author[J. Ortega-Cerd\`{a}]{Joaquim Ortega-Cerd\`{a}} 
\address{Department de Matem\`{a}tiques i Inform\`{a}tica, Universitat de 
Barcelona, Gran Via 585, 08007 Barcelona, Spain and
CRM, Centre de Recerca Matemàtica, Campus de Bellaterra Edifici C, 08193 
Bellaterra, Barcelona, Spain} 
\email{jortega@ub.edu}

\author[D. Radchenko]{Danylo Radchenko}
\address{Laboratoire Paul Painlev\'e \\ Universit\'e de Lille \\ F-59655 Villeneuve d'Ascq \\ France}
\email{danradchenko@gmail.com}

\author[K. Seip]{Kristian Seip} 
\address{Department of Mathematical Sciences, Norwegian University of Science and Technology (NTNU), 7491 Trondheim, Norway} 
\email{kristian.seip@ntnu.no}

\thanks{Bondarenko and Seip were supported in part by Grant 334466 of the 
Research Council of Norway. Ortega-Cerd\`{a} has been partially 
supported by grants  PID2021-123405NB-I00, PID2021-123151NB-I00 by the 
Ministerio de Ciencia e Innovación and by the Departament de Recerca i 
Universitats, grant 2021 SGR 00087. Radchenko acknowledges funding by the 
European Union (ERC, FourIntExP, 101078782).}

\begin{abstract}
We study the function $\varphi_1$ of minimal $L^1$ norm among all functions $f$ of exponential type at most $\pi$ for which $f(0)=1$. This function, first studied by H\"{o}rmander and Bernhardsson in 1993, has only real zeros $\pm \tau_n$, $n=1,2, \ldots$, and the sequence $(\tau_n-n-\frac12)$ has $\ell^2$ norm bounded by $0.13$. The zeros $\tau_n$ can be computed by means of a fixed point iteration. 
\end{abstract}

\maketitle
\section{Introduction}

In  \cite{HB93}, H\"{o}rmander and Bernhardsson obtained the remarkable numerical approximation 
\begin{equation}\label{eq:hb} 
	0.5409288219 \leq \mathscr{C}_1 \leq 0.5409288220,
\end{equation}
where $\mathscr{C}_p$ for $0<p<\infty$ is the smallest positive constant $C$ such that the inequality 
\begin{equation}\label{eq:pointeval} 
	|f(0)|^p \leq C \|f\|_{L^p(\mathbb R)}^p 
\end{equation}
holds for every $f$ in the Paley--Wiener space $PW^p$, i.e., the subspace of $L^p(\mathbb{R})$ consisting of entire functions of exponential type at most $\pi$. The problem of computing $\mathscr{C}_p$ for all $p$ arose in \cite{LL15, Lubinsky17} and was studied in some depth in the recent paper \cite{BCOS}. We refer to \cite{CMS19} for a number theoretic motivation for the study of the case $p=1$. 

Surprisingly, in spite of the impressive precision of  \eqref{eq:hb}, H\"{o}rmander and Bernhardsson obtained very limited information about the associated extremal function, i.e., the unique function that solves
\begin{equation}\label{eq:extremalproblem} 
	\frac{1}{\mathscr{C}_p} = \inf_{f \in PW^p} \left\{\|f\|_{L^p(\mathbb R)}^p \,: \, f(0) = 1\right\}
\end{equation}
in the case $p=1$. A compactness argument shows that \eqref{eq:extremalproblem} has solutions for any $0<p<\infty$, and a rescaling argument reveals that the type of these extremal functions is exactly~$\pi$. Uniqueness holds in the convex range $1 \leq p < \infty$ in which case the corresponding extremal function $\varphi_p$ is of the form
\[ \varphi_p(x) = \prod_{n=1}^\infty \left(1-\frac{x^2}{t_n^2}\right),\]
with $(t_n)_{n\geq1}$ a strictly increasing sequence of positive numbers (see \cite{HB93, LL15, BCOS}). Note that the case $p=2$ is trivial; then 
$\varphi_2(x)= \sinc(\pi x)$ (the reproducing kernel of $PW^2$ at $0$), where as usual $\sinc{x} \coloneqq  \frac{\sin{x}}{x}$. Hence $\mathscr{C}_2 = 1$ and $t_n=n$ for all $n$.

H\"{o}rmander and Bernhardsson's method of proof strongly suggests that $t_n$ is asymptotically close to $n+\frac12$ when $p=1$. The following theorem confirms that this is indeed the case; here and in what follows, $\pm \tau_n$ denotes the zeros of H\"{o}rmander and Bernhardsson's extremal function $\varphi_1$.

\begin{theorem}\label{thm:HB}
The $\ell^2$ norm of $(\tau_n-n-\frac12)$  is bounded by $0.13$. 

\end{theorem}

 Our proof gives an algorithm for computing $\tau_n$ in terms of a fixed point 
iteration. We have implemented this algorithm numerically and found, after 15 iterations, that for instance the first four zeros 
are close to 1.4417, 2.4657, 3.4756, 4.4811.
 
Little has been known about these zeros so far. H\"{o}rmander and Bernhardsson proved that all the zeros $\tau_n$ are simple; it was shown in \cite{BCOS} that $t_1\ge \frac{1}{2}$ and that the zero set of $\varphi_p$ is uniformly discrete for all $p\ge 1$. Some additional bounds on the separation of the zeros were also obtained in the range $p>2$.

We expect that our method of proof will enable us to say considerably more about the zeros and hence the extremal function $\varphi_1$. We plan to pursue our study of $\varphi_1$ in a forthcoming publication, which will also explore the relation to H\"{o}rmander and Bernhardsson's work. 

The rest of this note is devoted to the proof of Theorem~\ref{thm:HB}. Briefly, our approach will be to start from an operator equation satisfied by the $\tau_n$. The linear part of the 
operator in question is an ``almost'' unitary Hilbert-type matrix. Thanks to the favorable properties of this matrix, we are able to show, by Banach's fixed point theorem, that our operator equation has a unique solution whose $\ell^2$ distance to the sequence $(n+\frac12)$ can be quantified. 
\section{The zeros \texorpdfstring{$\tau_n$}{tn} as the solution of an 
equation}

We will use the notation $\| f \|_p\coloneqq \| f\|_{L^p(\mathbb R)}$ which we declare to be the $PW^p$ norm of $f$ should $f$ belong to $PW^p$. We start from the well known fact that a nontrivial function $\varphi$ in $PW^p$ is the extremal function $\varphi_p$ if and only if
\begin{equation} \label{eq:kernel}  f(0)=\int_{-\infty}^\infty f(x) \frac{|\varphi(x)|^{p-2}}{\|\varphi\|_{p}^p} \varphi(x) dx \end{equation}
holds for every $f$ in $PW^p$ (see \cite[Thm. 36]{BCOS}). Note that \eqref{eq:kernel} embodies the classical fact that $\varphi_2(x-y)$ is the reproducing kernel of $PW^2$ at $y$. 

It is clear from \eqref{eq:kernel} that the case $p=1$ is also particularly favorable. Indeed, when $\varphi$ is real with only real zeros, the function $|\varphi|^{p-2}\varphi $ will only take the values $\pm 1$ and change sign at the zeros of $\varphi$. In other words, the latter function depends in a simple explicit way on the location of the zeros of $\varphi$. 

Now let $\varphi$ be a function in $PW^1$ of the form
\begin{equation} \label{eq:prod} \varphi(x) = \prod_{n=1}^\infty \left(1-\frac{x^2}{t_n^2}\right), \quad 0<t_1 < t_2 < \cdots, \end{equation}
  and set
\[ \Phi(x)\coloneqq \frac{\varphi(x)}{|\varphi(x)| \|\varphi\|_{1}}. \]
Following an idea in \cite[p. 188]{HB93}, we will  relate $\Phi$ to the function 
\[ \Psi(x)\coloneqq - \| \varphi \|_1^{-1} \operatorname{sgn} \cos \pi x .\] 
We define the Fourier transform of an $L^1$ function $f$ as
\[ \widehat{f}(\xi)\coloneqq \int_{-\infty}^\infty f(x) e^{-i\xi x} dx \] and find that the distributional Fourier transform of $\Psi$ is
\[ \widehat{\Psi}(\xi)=  \pi \widehat{g}(\xi)\sum_{n=-\infty}^{\infty} \delta_{\pi n}(\xi), \]
where $\widehat{g}$ is the Fourier transform of $\chi_{(-\pi,\pi)}\Psi$. Since $\widehat{g}(0)=0$, we see  that  $\widehat{\Psi}$ vanishes on $(-\pi,\pi)$, as observed in \cite[p. 188]{HB93}.

\begin{lemma} Let $\varphi$ be of the form \eqref{eq:prod} and assume that $(t_n-n-\frac12)$ is in $\ell^2$. Then $(\Phi-\Psi)\ast \varphi_2$ is in $PW^2$ and
\begin{equation} \label{eq:charext}  \frac{1}{\|\varphi \|_{1}}  \le \mathscr{C}_1 \le \frac{1}{\|\varphi \|_{1}} + \| \varphi_2 -(\Phi-\Psi)\ast \varphi_2 \|_{\infty}.  \end{equation}  \end{lemma}
\begin{proof} We begin by noting that
\[ \frac{((\Phi-\Psi)\ast \varphi_2)(y)}{2\|\varphi\|_{1}}= \int_{-\frac{1}{2}}^{\frac{1}{2}} \varphi_2(x-y) dx + \sum_{n=1}^{\infty} (-1)^n \int_{t_n}^{n+\frac{1}{2}} \left( \varphi_2(x-y) + \varphi_2(x+y)\right) dx. \]
By duality, we therefore find that
\[ \frac{\|(\Phi-\Psi)\ast \varphi_2\|_{2}}{2\|\varphi\|_{1}}
=\inf_{\substack{f\in PW^2\\ \|f\|_2\le 1}} \left| \int_{-\frac{1}{2}}^{\frac{1}{2}} f(x)dx + \sum_{n=1}^{\infty} (-1)^n \int_{t_n}^{n+\frac{1}{2}} \left(f(x)+f(-x)\right) dx\right|. \]
Here we may assume that $f$ is real-valued, and so the expression on the right-hand side can be written as 
\[ f(x_0)+ \sum_{n=1}^{\infty} (-1)^n \big(n+\frac{1}{2}-t_n\big)\left(f(x_n) + f(x_{-n})\right) \]
by the mean value theorem, where $|x_0| < \frac12$ and $x_{\pm n}$ lie between $\pm(n+\frac12)$ and $\pm t_n$. Now 
applying the Cauchy--Schwarz inequality to the sum and invoking the Plancherel--P\'{o}lya inequality (see e.g. \cite[Lect. 7]{Levin96}), we conclude that $(\Phi-\Psi)\ast \varphi_2$ is in $PW^2$.

The left inequality in \eqref{eq:charext} is trivial. To prove the right inequality, we set
\[ F\coloneqq \varphi_2 + \Phi - (\Phi-\Psi)\ast \varphi_2 \]
and let $f$ be a function in $PW^1\cap \mathscr{S}$, where $\mathscr{S}$ is the Schwartz space. Since $PW^1\cap \mathscr{S}$ is dense in $PW^1$, we have 
\[ \mathscr{C}_1=\sup_{\substack{f\in PW^1\cap \mathscr S\\ \| f\|_{1}=1}} |f(0)|. \]
But since $\widehat{\Psi}$ vanishes on $(-\pi,\pi)$, we also have $f(0)=(F\ast f)(0)$, and so
\[ \mathscr{C}_1 \le \| \varphi\|_1^{-1} +  \| \varphi_2 -(\Phi-\Psi)\ast \varphi_2 \|_{\infty} \]
by the triangular inequality for the $L^\infty(\mathbb R)$ norm.
\end{proof}
It follows from the above lemma that $\varphi=\varphi_1$ if $(\Phi-\Psi)\ast \varphi_2=\varphi_2$. In fact, we may observe that this would still hold if we merely had $(\Phi-\Psi)\ast \varphi_2=c\varphi_2$
for some constant $c$. Indeed, the Fourier transform of $\Phi-\Psi$ would then be constant on $(-\pi,\pi)$, and so
\[ \int_{-\infty}^\infty ((\Phi-\Psi)\ast \varphi_2)(x)\varphi(x) dx=\int_{-\infty}^\infty (\Phi(x)-\Psi(x))\varphi(x) dx = \int_{-\infty}^\infty \Phi(x)\varphi(x) dx=\varphi(0),\] 
whence $\varphi_2 - (\Phi-\Psi)\ast \varphi_2=0$, and consequently $c=1$.
Since a function in $PW^2$ is determined by its values at the integers, we would therefore obtain $\tau_n$ as $n+\frac12-\delta_n$ if 
\begin{equation} \label{eq:start}  \int_{-\frac{1}{2}}^{\frac{1}{2}} \sinc \pi(x-k) dx + \sum_{n=1}^{\infty} (-1)^n \int_{n+\frac12-\delta_n}^{n+\frac{1}{2}} \left( \sinc \pi (x-k) + \sinc \pi (x+k)\right) dx=0 \end{equation}
for all positive integers $k$ with $\delta\coloneqq (\delta_1, \delta_2, \ldots )$ in $\ell^2$. We prefer to rewrite the left-hand side of this equation by 
making the following integration by parts:
\[ \int_{-\frac{1}{2}}^{\frac{1}{2}} \sinc \pi(x-k) dx = \frac{(-1)^{k+1}}{\pi^2} \int_{-\frac{1}{2}}^{\frac{1}{2}} \frac{\cos \pi x}{(x-k)^2} dx .\]
Then \eqref{eq:start} takes the form
\begin{equation} \label{eq:equiv}   \sum_{n=1}^{\infty} \int_0^{\delta_n} \cos \pi y \Big(\frac{1}{n+\frac{1}{2}-y-k}+\frac{1}{n+\frac{1}{2}-y+k}\Big)dy=\frac{1}{\pi}\int_{-\frac{1}{2}}^{\frac{1}{2}} \frac{\cos \pi x}{(x-k)^2} dx \end{equation}
for $k=1,2, \ldots$. The above discussion shows that Theorem~\ref{thm:HB} is equivalent to the following assertion. (Here and in what follows, $\|\cdot\|$ denotes the $\ell^2$ norm.)
\begin{theorem} \label{thm:eqHB} Equation \eqref{eq:equiv} has a solution $\delta$ satisfying $\| \delta\|\le 0.13$. \end{theorem}
 The next two sections will give the proof of this theorem.
 \section{The linear part \texorpdfstring{of \eqref{eq:equiv}}{}}
 
 Performing a linear approximation of the operator on the left-hand side of \eqref{eq:equiv}, we may write this equation as
\begin{equation} \label{eq:AQ}  A\delta + Q\delta = w ,\end{equation}
where
\begin{align*} w_k & \coloneqq  \frac{1}{\pi}\int_{-\frac{1}{2}}^{\frac{1}{2}} \frac{\cos \pi x}{(x-k)^2} dx; \\
     (A\delta )_k & \coloneqq \sum_{n=1}^{\infty} \left(\frac{1}{n+k+\frac12}+\frac{1}{n-k+\frac12}\right)\delta_n ; \\ 
     (Q\delta)_k & \coloneqq  \sum_{n=1}^{\infty} \int_0^{\delta_n} \Big(\frac{\cos \pi y}{n+\frac{1}{2}-y+k}-\frac{1}{n+\frac12+k}+\frac{\cos \pi y}{n+\frac{1}{2}-y-k}-\frac{1}{n+\frac12-k}\Big)dy.\end{align*}
We will use the notation $a_{k,n}\coloneqq \frac{1}{n+k+1/2}+\frac{1}{n-k+1/2}$ for $n\ge1$ so that 
\[ (Ax)_k=\sum_{n=1}^{\infty} a_{k,n} x_n. \]
The proof of Theorem~\ref{thm:eqHB} will rely on the fact that the infinite matrix $\frac{1}{\pi}A$ is ``almost'' unitary. 
\begin{lemma}\label{lem:norm}
 	The operator $A$ maps $\ell^2$ to $\ell^2$ and satisfies
 	\[2\sqrt{2}\|x\|\le \|Ax\|\le \pi\|x\|\]
 	for all $x$ in $\ell^2$. 
\end{lemma}

The proof of Lemma~\ref{lem:norm} will employ the following fact. 
\begin{lemma}\label{lem:partialf}
For $r\ge2$, $a_1,\dots,a_r\in\Z$, set
	\[C(a_1,\dots,a_r):=\sum_{n\in\Z}\frac{1}{(n+a_1+\frac12)\dots(n+a_r+\frac12)}\,.\]
Then $C(a_1,\dots,a_r)=0$ whenever $a_1,\dots,a_{r}$ are distinct integers. Moreover, we
have
\[C(a,a)=\pi^2\,,\qquad C(a,b,b)=\frac{\pi^2}{a-b}\,,\qquad
	C(a,a,b,b)=\frac{2\pi^2}{(a-b)^2}\,.\]
\end{lemma}
\begin{proof} The result is immediate by a partial fraction decomposition of the summands.~\end{proof} 

\begin{proof}[Proof of Lemma~\ref{lem:norm}]
	Using Lemma~\ref{lem:partialf},
	we get
	\begin{equation*}
	\frac{2}{(n+\frac12)(m+\frac12)}+\sum_{k=1}^{\infty}a_{k,n}a_{k,m}\\
	=C(n,m)-C(n,-m) = \pi^2\delta_{n,m}\,.
	\end{equation*}
	Therefore
	\begin{equation*}
	\|Ax\|^2 +2\Big(\sum_{n\ge1}\frac{x_n}{n+\frac12}\Big)^2 =  \pi^2\|x\|^2\,,
	\end{equation*}
	which immediately implies that $\|Ax\|\le \pi\|x\|$. To get the lower bound $\|Ax\|\ge 2\sqrt{2}\|x\|$, we apply the Cauchy--Schwarz inequality to the second term on the right-hand side and use that
	\[ \sum_{n=1}^{\infty} \frac{1}{(n+\frac12)^2}=\frac{\pi^2}{2}-4. \qedhere \]
\end{proof}
Note that the above proof reveals the interesting fact that $\frac1\pi A$ restricted to the orthogonal complement of the sequence $((n+\frac12)^{-1})$ is an isometry. 

We next show that $A$ is invertible by identifying its inverse $B$. Lemma~\ref{lem:norm} shows that $\| Bx\|\le (2\sqrt{2})^{-1} \|x\|$, and this is the only information needed about $B$ in the proof of Theorem~\ref{thm:eqHB}. However, it is useful for our numerical work 
to have an explicit expression for~$B$. 
\begin{proposition} 
The operator $B$ with matrix
	\[b_{n,k} =  \frac{1-(2n+1)^{-2}}{\pi^2(1-(2k)^{-2})}\, a_{k,n}\]
satisfies $AB=BA=I$.
\end{proposition}
\begin{proof}
	Note that $a_{k,n}=a_{-k,n}=-a_{k,-1-n}$. Using this, we find that
	\begin{align*}
	\sum_{n=1}^{\infty}&a_{k_1,n}a_{k_2,n}(1-(2n+1)^{-2}) =
	\frac12\sum_{n\in\Z}a_{k_1,n}a_{k_2,n}(1-(2n+1)^{-2})\\
	&=\frac12\Big(C(k_1,k_2)+C(k_1,-k_2)+C(-k_1,k_2)+C(-k_1,-k_2)-C(k_1,k_2,-k_1,-k_2)\Big)\,,
	\end{align*}
	which clearly vanishes for $k_1\ne k_2$, and for $k_1=k_2=k$ we get
	$\pi^2(1-(2k)^{-2})$, from which we see that $BA=I$. Similarly, we calculate 
	\begin{align*}
	\sum_{k=1}^{\infty}&\frac{a_{k,m}a_{k,n}}{1-(2k)^{-2}} =
	\frac12\sum_{k\in\Z}a_{k,m}a_{k,n}\Big(1+\frac{\frac14}{k-\frac12}-\frac{\frac14}{k+\frac12}\Big)\\
	=\frac12\Big(&C(m,n)-C(m,-1{-}n)-C(-1{-}m,n)+C(-1{-}m,-1{-}n)\Big)\\
	+\frac18\Big(&C(-1,m,n)-C(-1,m,-1{-}n)-C(-1,-1{-}m,n)+C(-1,-1{-}m,-1{-}n)\Big)\\
	-\frac18\Big(&C(0,m,n)-C(0,m,-1{-}n)-C(0,-1{-}m,n)+C(0,-1{-}m,-1{-}n)\Big)\,,
	\end{align*}
	which after simplifications is easily seen to imply $AB=I$.
\end{proof}
\section{Proof of Theorem~\ref{thm:eqHB}}
Retaining the notation $B$ for the inverse of $A$ and setting $x\coloneqq \delta - Bw$, we see from \eqref{eq:AQ} that our task is to find a fixed point of the  
map 
\[ Gx\coloneqq -BQ(x+Bw). \] 
We will prove that a unique solution to this problem can be found in the ball
\[ K\coloneqq \Big\{ x=(x_k): \  \| x \| \le 0.042 \Big\}.\]

\begin{theorem}\label{thm:contractive}
$G$ maps $K$ into itself and $\| G(x)-G(y) \| \le 0.73\, \| x-y\|$ for all $x,y$ in $K$. 
\end{theorem}

In view of Banach's fixed point theorem, Theorem~\ref{thm:eqHB} follows from Theorem~\ref{thm:contractive} and  the next lemma.

\begin{lemma}
We have $\| Bw \|\le 0.088$. \end{lemma}
\begin{proof}
We see that 
\[ \frac{1}{\pi} \int_{-\frac{1}{2}}^{\frac{1}{2}} \frac{\cos \pi x}{(x-n)^2} dx \le 
\frac{1}{\pi n^2} \int_{-\frac{1}{2}}^{\frac{1}{2}} \frac{\cos \pi x}{(1-x)^2} dx ,\]
and so 
\[ \| w \|_2 \le  \frac{\sqrt{\zeta(4)}}{\pi} \int_{-\frac{1}{2}}^{\frac{1}{2}} \frac{\cos \pi x}{(x-1)^2} dx. \]
By Lemma~\ref{lem:norm}, we then get the desired bound
\[ \| Bw \|_2 \le  \frac{\pi}{12\sqrt{5} } \int_{-\frac{1}{2}}^{\frac{1}{2}} \frac{\cos \pi x}{(x-1)^2} dx  \le 0.088. \qedhere \]
\end{proof}


The proof of Theorem~\ref{thm:contractive} splits naturally into the following two steps. 
\begin{proof}[Step 1: Proof that $G$ maps $K$ into itself]
 We set $\xi\coloneqq x+Bw$ and $\varepsilon\coloneqq \|\xi\|$. We begin by noting that
\[ \frac{\cos\pi y}{t-y}=\frac{1}{t}\left(1+\frac{y}{t}+\frac{\frac{y^2}{t^2}}{1-\frac{y}{t}}\right)\cos\pi y, \]
whence
\[\frac{\cos\pi y}{t-y}-\frac{1}{t}=\frac{\cos \pi y-1}{t}+\frac{1}{t}\left(\frac{y}{t}+\frac{\frac{y^2}{t^2}}{1-\frac{y}{t}}\right)\cos\pi y. \]
If we set
\begin{equation}  \alpha_n\coloneqq \int_0^{\xi_n} \big(\cos \pi y-1\big) dy = 
\pi^{-1}\sin (\pi \xi_n)-\xi_n \end{equation}
and 
\[ \beta_{n,k}^{\pm}\coloneqq  \int_{0}^{\xi_n} \left(1+\frac{\frac{y}{n+\frac12\pm k}}{1-\frac{y}{n+\frac12\pm k}}\right)y\cos\pi y dy ,\] 
then we may write
\[ (Q\xi)_k = (A\alpha)_k + \sum_{n=1}^\infty \left(\frac{\beta_{n,k}^{+}}{(n+\frac12+k)^2} +\frac{\beta^{-}_{n,k}}{(n+\frac12-k)^2}\right). \]
We see that
\[ |\beta_{n,k}^{\pm}|\le \frac{\xi_n^2}{2}+\frac{2|\xi_n|^3}{3(1-2|\xi_n|)}\le \left(\frac12+\frac{2\varepsilon}{3(1-2\varepsilon)}\right) \xi_n^2=:c(\varepsilon) \xi_n^2. \] 
Setting $D\coloneqq (d_{n,k})$ with $d_{n,k}\coloneqq \frac{1}{(n+\frac12+k)^2} +\frac{1}{(n+\frac12-k)^2}$, we therefore find that
\begin{equation} \label{eq:BQe}  \| BQ \xi \| \le \| \alpha \| + (2\sqrt{2})^{-1} c(\varepsilon) \| D (\xi_n^2)\|.\end{equation}
Since
\[ |\alpha_n|\le \frac{\pi^2}{6} |\xi_n|^3 ,\]
we have
\begin{equation} \label{eq:alpha}  \| \alpha \|\le \zeta(2) \varepsilon^3. \end{equation}
By duality and an application of the Cauchy--Schwarz inequality,
\begin{equation}  \label{eq:Db} \|D (\xi_n^2) \|  \le \sup_{q:\|q\|=1} \sum_{k=1}^\infty \sum_{n=1}^\infty \left(\frac{1}{(n+\frac12+k)^2} +\frac{1}{(n+\frac12-k)^2} \right) q_k \xi_n^2  
    \le C \| (\xi_n^2) \| \le C \varepsilon^2, \end{equation}
   where 
   \[ C\coloneqq \sum_{n=2}^\infty \frac{1}{(n+\frac12)^2}+2\sum_{k=0}^\infty \frac{1}{(k+\frac12)^2} \le \frac{1}{2}+6\zeta(2).\]
   Inserting \eqref{eq:alpha} and \eqref{eq:Db} into \eqref{eq:BQe}, we conclude that
   \[  \| BQ \xi \| \le \Big(\zeta(2) \varepsilon + \frac{3}{\sqrt{2}} c(\varepsilon) \big(\zeta(2)+\frac{1}{12}\big)  \Big) \varepsilon^2 \le 0.042 \] 
   if $\varepsilon\le 0.13$. This gives the desired result since $\|Bw\|\le 0.088$.
 \end{proof}
\begin{proof}[Step 2: Proof of the contractivity of $G$] Let $x$ and $z$ be two arbitrary points in $K$, and set $\xi\coloneqq x+Bw$ and $\eta\coloneqq z+Bw$. We set
$\varepsilon\coloneqq \max(\| \xi\|, \| \eta \|)$ and argue as in the preceding proof to get
\[ (Q\xi)_k - (Q\eta)_k = (A\alpha)_k + \sum_{n=1}^\infty \left(\frac{\beta_{n,k}^+}{(n+\frac12+k)^2} +\frac{\beta_{n,k}^-}{(n+\frac12-k)^2}\right), \]
where now
\[ \alpha_n\coloneqq  \int_{\eta_n}^{\xi_n} \big(\cos \pi y-1\big) dy = \pi^{-1}\big(\sin (\pi \xi_n)-\sin (\pi \eta_n)-(\xi_n-\eta_n)\big) \]
and 
\[ \beta_{n,k}^{\pm}= \int_{\eta_n}^{\xi_n} \left(1+\frac{\frac{y}{n+\frac12\pm k}}{1-\frac{y}{n+\frac12\pm k}}\right)y\cos\pi y dy .\] 
In the present case, we have
\[ |\alpha_n| \le \frac{\pi^2}{6}|\xi_n^3-\eta_n^3|\le \frac{\pi^2\varepsilon^2}{2}|x_n-z_n| \]
 and
 \[ |\beta_{n,k}^{\pm}| \le \frac{|\xi_n^2-\eta_n^2|}{2} +\frac{2|\xi_n^3-\eta_n^3|}{3(1-2\varepsilon)}\le \left(\varepsilon+\frac{2\varepsilon^2}{1-2\varepsilon}\right)  |x_n-z_n|.\]
 Estimating as before, we get that 
 \[ \| BQ \xi - BQ \eta \| \le  \left(\frac{\pi^2\varepsilon^2}{2}+\frac{\big(\frac{1}{2}+\pi^2\big)}{2\sqrt{2}}\big(\varepsilon+\frac{2\varepsilon^2}{1-2\varepsilon}\big)   \right) \| x-z\|, \] 
 where 
 \[ \frac{\pi^2\varepsilon^2}{2}+\frac{\big(\frac{1}{2}+\pi^2\big)}{2\sqrt{2}}\big(\varepsilon+\frac{2\varepsilon^2}{1-2\varepsilon}\big)\le 0.73\]
 when $\varepsilon\le 0.13$.
 \end{proof}

\bibliographystyle{amsplain} 
\bibliography{pwpeval}

\end{document}